\documentclass[a4paper]{article}
\usepackage{hyperref}
\usepackage{amsmath}
\usepackage{amsthm}
\usepackage{amsfonts}
\usepackage{mathrsfs}
\usepackage{graphicx}
\usepackage{epic,eepic}
\graphicspath{{./fig/}}

\def\R{\mathbf{R}}

\def\W{\mathrm{W}}
\def\L{\mathrm{L}}
\def\C{\mathbf{C}}

\def\SS{\mathbf{S}}

\def\CC{\mathscr{C}}
\def\LL{\mathscr{L}}

\def\Sp{\mathrm{Sp}}
\def\iy{\infty}
\renewcommand\d{\mathrm{d}}
\def\veps{\varepsilon}
\newcommand\ad{\mathrm{ad}}
\newcommand\acos{\mathrm{acos}}
\renewcommand\vec{\overrightarrow}
\renewcommand\bar{\overline}
\renewcommand\tilde{\widetilde}

\def\resp{\emph{resp.}}
\def\eg{\emph{e.g.}}

\def\noi{\noindent}
\def\t{\!\:^t}
\newcommand{\frp}[2]{\frac{\partial #1}{\partial #2}}
\newcommand{\frpp}[2]{{\partial #1}/{\partial #2}}
\newtheorem{thrm}{Theorem}
\newtheorem{prpstn}{Proposition}
\newtheorem{lmm}{Lemma}
\newtheorem{crllr}{Corollary}
\newtheorem*{schlm}{Scholium}
\theoremstyle{definition}
\theoremstyle{remark}
\newtheorem{rmrk}{Remark}
\title{\bf $\mathbf{L^1}$-minimization for mechanical systems}
\author{Z.~Chen\thanks{Math.\ Dep., Univ.\ Paris-Sud \& CNRS and
Northwestern Polytechnical Univ. (\texttt{zheng.chen@math.u-psud.fr}).
Supported by Chinese Scholarship Council (grant no.\ 2013 0629 0024).}
\and J.-B.~Caillau\thanks{Math.\ Institute, Univ.\ Bourgogne \& CNRS/INRIA
(\texttt{jean-\-baptiste.caillau@u-bourgo gne.fr}). Part of this work was done during
a sabbatical leave at Lab.\ J.-L.\ Lions, Univ.\ Paris VI \& CNRS, whose hospitality
is gratefully acknowledged.}
\and Y.~Chitour\thanks{L2S-Supelec, Univ.\ Paris-Sud \& CNRS
(\texttt{yacine.chitour@lss.supelec.fr}).}}
\date{March 2015}

\begin{document}
\maketitle

\begin{abstract} \noi
Second order systems whose drift is defined by the gradient of a given potential are
considered, and minimization of the $\L^1$-norm of the control is addressed.
An analysis of the extremal flow emphasizes the role of singular
trajectories of order two \cite{robbins-1965a,zelikin-1994a};
the case of the two-body potential is treated in detail.
In $\L^1$-minimization, regular extremals are associated with controls whose norm is
bang-bang; in order to
assess their optimality properties, sufficient conditions are given for broken
extremals and related to the no-fold conditions of \cite{schattler-2002a}. An example of
numerical verification of these conditions is proposed on a problem coming from space
mechanics.\\

\noi\textbf{Keywords.} $\L^1$-minimization, second order mechanical systems,
order two singular trajectories, no-fold conditions for broken extremals, two-body problem\\ 

\noi\textbf{MSC classification.} 49K15, 70Q05\\
\end{abstract}

\pagestyle{myheadings}
\markboth{}{$L^1$-minimization for mechanical systems}

\section{Introduction} \label{s1}
This paper is concerned with the optimal control of mechanical systems of the
following form:
\[ \ddot{q}(t) + \nabla V(q(t)) = \frac{u(t)}{M(t)}\,,\quad
   \dot{M}(t) =-\beta |u(t)|, \]
where $q$ is valued in an open subset $Q$ of $\R^m$, $m \geq 2$,
on which the potential $V$ is defined.
The second equation describes the variation of the
mass, $M$, of the system when a control is used ($\beta$ is some nonnegative constant).
The finite dimensional norm is Euclidean,
\[ |u| = \sqrt{u_1^2+\cdots+u_m^2} \]
and a constraint on the control is assumed,
\begin{equation} \label{eq1.3}
  |u(t)| \leq \veps,\quad \veps > 0.
\end{equation}
Given boundary conditions in the $n$-dimensional state (phase) space $X:=TQ
\simeq Q\times\R^m$ ($n=2m$), the problem of interest is the minimization of
consumption, that is the maximization of the final mass $M(t_f)$ for a fixed final
time. Clearly, this amounts to minimizing the $\L^1$-norm of the control,
\begin{equation} \label{eq1.1}
  \int_0^{t_f} |u(t)|\,\d t \to \min.
\end{equation}
Up to some rescaling, there are actually two cases, $\beta=1$ or $\beta=0$. In the
second one, the mass is constant; though maximizing the final mass does not make sense
anymore, the Lagrange cost (\ref{eq1.1}) is still meaningful. Actually,
as propellant is only a limited fraction of the total mass, one can expect this
idealized constant mass model to capture the main features of the original problem.
We shall henceforth assume $\beta=0$, so the state reduces to $x:=(q,v)$ with
$v:=\dot{q}$.

In finite dimensions, $\ell^1$-minimization is well-known to generate sparse
solutions having a lot of zero components; this fact translates here into the
existence of subintervals of time where the control vanishes, as is clear when
applying the maximum principle (see \S\ref{s2}). This intuitively goes
along well with the idea of minimizing consumption: There are privileged values of the
state where the control is more efficient and should be switched on (\emph{burn} arcs),
while there are
some others where it should be switched off (\emph{cost} arcs).
(See also \cite{berret-2008a} for a different kind of interpretation in a biological
setting, again with $\L^1$-minimization.)
The resulting sparsity of the solution is then
tuned by the ratio of the fixed final time over the minimum time associated with the
boundary conditions: While a simple consequence of the form of the dynamics (and of
the ball constraint on the control) is that the min.\ time control norm is constant
and maximum everywhere for the constant mass model,\footnote{See, \eg,
\cite{cocv-2001}; this fact remains true for time minimization if the mass is
varied provided the mass at final time is left free [\emph{ibid}].}
the extra amount of time available allows for some optimization that results in the
existence of subarcs of the trajectory with zero control.
(See Proposition~\ref{prop2.1}, in this respect.)
A salient peculiarity of the infinite dimensional setting is the existence of
subarcs with intermediate value of the norm of the control, namely singular arcs.
This was analyzed in the seminal paper of Robbins \cite{robbins-1965a} in the case of
the two-body potential, providing yet another example of the fruitful exchanges between
space mechanics and optimal control in the early years of both disciplines. The
consequence of these singular arcs being of order two was further realized by Marchal
who studied chattering in \cite{marchal-1973a}; this example comes probably second
after the historical one of Fuller \cite{fuller-1964a}
and has been thoroughly
investigated by Zelikin and Borisov in \cite{zelikin-1994a,zelikin-2003a}.

A typical example of second order controlled system is the restricted
three-body problem \cite{sicon-2012} where, in complex notation ($\R^2 \simeq \C$),
\[ V_\mu(t,q) := -\frac{1-\mu}{|q+\mu e^{it}|}-\frac{\mu}{|q-(1-\mu)e^{it}|}\cdot \]
In this case, $\mu$ is the ratio of the masses of the two primary celestial bodies,
in circular motion around their common center of mass. The controlled third body is a
spacecraft gravitating in the potential generated by the two primaries, but not
influencing their motion. When $\mu=0$, the potential is autonomous and one retrieves
the standard controlled two-body problem. The study of "continuous" (as opposed to
impulsive) strategies for the control began in the 60's; see, \eg, the work of Lawden
\cite{lawden-1961a}, or Beletsky's book \cite{beletsky-1999a}
where the importance of low thrust (small $\veps$ in (\ref{eq1.3})) to spiral
out from a given initial orbit was foreseen. There is currently a strong interest for
low-thrust missions with, \eg, the Lisa Pathfinder \cite{lisa}
one of ESA\footnote{European Space Agency.} towards the $L_1$ Lagrange
point of the Sun-Earth system, or BepiColombo \cite{bepicolombo} mission of ESA and
JAXA\footnote{Japan Aerospace Exploration Agency.} to Mercury.

An important issue in optimal control is the ability to verify sufficient optimality
conditions. In $\L^1$-minimization, the first candidates for optimality are controls
whose norm is bang-bang, switching from zero to the bound prescribed by (\ref{eq1.3})
(more complicated situations including singular controls). Second order conditions in the
bang-bang case have received quite an extensive treatment; references include
the paper of Sarychev \cite{sarychev-1997a},
followed by \cite{agrachev-2002a} and
\cite{maurer-2004a,osmolovskii-2005a,osmolovskii-2007a}. On a similar line, the
stronger notion of state optimality was introduced in \cite{poggiolini-2004a} for free
final time. More recently, a regularization procedure has been developed in
\cite{trelat-2010b} for single-input systems. 
These papers consider controls valued in polyhedra; the standing
assumptions allow to define a finite dimensional accessory optimization problem in the
switching times only. Then, checking a second order sufficient condition on this
auxiliary problem turns to be sufficient to ensure strong local optimality of the
bang-bang controls. A byproduct of the analysis is that conjugate times, where local
optimality is lost, are switching times. A different approach, based on
Hamilton-Jacobi-Bellman and the method of characteristics in optimal control, has been 
proposed by Noble and Sch\"attler in \cite{schattler-2002a}. Their results encompass
the case of broken extremals
with conjugate points occuring at \emph{or} between switching times. We provide a similar
analysis by requiring some generalized (with respect to the smooth case) disconjugacy
condition on the Jacobi fields, and using instead a Hamiltonian point of view
reminiscent of \cite{ekeland-1977a,kupka-1987a}. Treating the case of such broken
extremals is crucial for $\L^1$-minimization: As the finite dimensional norm
of the control involved in the constraint (\ref{eq1.3}) and in the cost (\ref{eq1.1})
is an $\ell^2$-norm, the control is valued in the Euclidean
ball of $\R^m$, not a polyhedron if $m>1$. When $m=1$, the situation is degenerate,
and one can for instance set $u=u_+-u_-$, with $u_+,u_-\geq 0$. 
(This approach also works for $m>1$ when an $\ell^1$ or $\ell^\iy$-norm is used for the
values of the control; see, \eg, \cite{vossen-2006a}.) When $m>1$, it is clear using
spherical coordinates that although the norm of
the control might be bang-bang, the variations of the control component on $\SS^{m-1}$
preclude the reduction to a finite dimensional optimization problem.
(The same remark holds true for any $\ell^p$-norm of the control values with
$1<p<\iy$.)
An example of conjugacy occuring between switching times is provided in \S\ref{s4}.

The paper is organized as follows. In section~\ref{s2}, the extremal lifts of
$\L^1$-minimizing trajectories are studied for an arbitrary potential in the constant
mass case; the properties of the flow are encoded by the Poisson structure defined by
two Hamiltonians. In section~\ref{s3}, sufficient conditions for strong local
optimality of broken extremals
with regular switching points are given in terms of jumps on the Jacobi fields; these
conditions are related to the no-fold condition of
\cite{schattler-2002a}. In section~\ref{s4} some numerical results illustrating the
verification of these sufficient conditions for $\L^1$-minimizing trajectories are
given. The two-body mechanical potential is considered, completing the study of
Gergaud and Haberkorn \cite{gergaud-2006a} where the first numerical computation of fuel
minimizing controls with hundreds of switchings (for low thrust) was performed using a
clever combination of shooting and homotopy techniques. (See also \cite{oberle-1977a}
in the case of a few switchings.)
The classical construction of 
fields of extremals in the smooth case is reviewed in an appendix.

\section{Singularity analysis of the extremal flow} \label{s2}
By renormalizing the time and the potential, one can assume $\veps=1$ in
(\ref{eq1.3}), so we consider the $\L^1$-minimum control of
\[ \ddot{q}(t)+\nabla V(q(t)) = u(t),\quad |u(t)| \leq 1, \]
with $x(t)=(q(t),v(t)) \in X=TQ$ ($v(t)=\dot{q}(t)$),
$Q$ an open subset of $\R^m$ ($m \geq 2$),
and make the following assumptions on the boundary conditions:
\[ x(0)=x_0 \quad \text{and} \quad x(t_f) \in X_f \subset X \]
where (i) $x_0$ does not belong to the terminal submanifold $X_f$, (ii) $X_f$ is
invariant wrt.\ the flow of the drift,\footnote{This assumption can be weakened; it is
only used to ensure that a time minimizing control extended by zero beyond the
min.\ time is admissible. (See Lemma~\ref{lem2.1}.)}
\[ F_0(q,v) = v\frp{}{q}-\nabla V(q)\frp{}{v}\,, \]
and (iii) the fixed final time $t_f$ is supposed strictly greater than the minimum
time $\bar{t}_f(x_0,X_f)<\iy$ of the problem.
As the cost is not differentiable for $u=0$, rather than using a non-smooth maximum
principle (compare, \eg, \cite{berret-2008a}) we make a simple desingularization: In
spherical coordinates, $u=\rho w$ where $\rho \in [0,1]$ and $w \in \SS^{m-1}$; the
change of coordinates amounts to adding an $\SS^{m-1}$ fiber above the singularity
$u=0$ of the cost. In these
coordinates, the dynamics write
\[ \dot{x}(t) = F_0(x(t))+\rho(t)\sum_{i=1}^m w_i(t) F_i(x(t)) \]
with canonical $F_i=\frpp{}{v_i}$, $i=1,\dots,m$, and the criterion is linearized:
\[ \int_0^{t_f} \rho(t)\,\d t \to \min. \]
The Hamiltonian of the problem is
\[ H(x,p,\rho,w) = p^0\rho+H_0(x,p)+\rho\sum_{i=1}^m w_i \psi_i(x,p) \]
where $H_0(x,p):=pF_0(x)$ and the
$\psi_i(x,p):=pF_i(x)$ are the Hamiltonian lifts of the $F_i$, $i=1,\dots,m$.
Readily, $H \leq H_0+\rho H_1$ with
\[ H_1 := p^0+\sqrt{\sum_{i=1}^m \psi_i^2}\,, \]
and the equality can always be achieved for some $w \in \SS^{m-1}$: $w=\psi/|\psi|$
whenever $\psi:=(\psi_1,\dots,\psi_m)$ is not zero, any $w$ on the sphere otherwise.
By virtue of the maximum principle, if $(\rho,w)$ is a measurable minimizing control
then the associated trajectory is the projection of an integral curve $(x,p):[0,t_f]
\to T^*X$ of $H_0+\rho H_1$ such that, a.e.,
\begin{equation} \label{eq2.1}
  H_0(x(t),p(t))+\rho(t)H_1(x(t),p(t)) =
   \max_{r \in [0,1]} H_0(x(t),p(t))+rH_1(x(t),p(t)).
\end{equation}
Moreover, the constant $p^0$ is nonpositive and $(p^0,p) \neq (0,0)$.
Either $p^0=0$ (abnormal case), or $p^0$ can be set to $-1$ by homogeneity (normal case).

\begin{prpstn}[Gergaud \emph{et al.}\ \cite{gergaud-2006a}] \label{prop2.1}
There are no abnormal extremals.
\end{prpstn}

\begin{lmm} \label{lem2.1}
The function $\psi$ evaluated along an extremal has only isolated zeros.
\end{lmm}

\begin{proof} 
As a function of time when evaluated along an extremal, $\psi$ is absolutely continous
and, a.e.\ on $[0,t_f]$,
\[ \dot{\psi}_i(t) = p(t)[F_0,F_i](x(t)),\quad i=1,\dots,m. \]
As a result, $\psi$ is a $\CC^1$ function of time and, if $\psi(t)=0$, then
$\dot{\psi}(t) \neq 0$; indeed, the rank of
$\{F_1,\dots,F_m,[F_0,F_1],\dots,[F_0,F_m]\}$ is maximum everywhere, so $p(t)$ would
otherwise be zero.
Since $p$ is solution of a linear ode, this would imply that it vanishes
identically; necessarily $p^0<0$, so $\rho$ would also be zero a.e.\ because of the
maximization condition (\ref{eq2.1}).
This is impossible because $x_0 \notin X_f$.
\end{proof}

\begin{proof}[Proof of the Proposition] By contradiction: Assume $p^0=0$; as $\psi$ has only isolated zeros
according to the previous lemma, $\rho=1$ a.e.\ by maximization. The resulting cost
is equal to $t_f$.
Now, the target
submanifold $X_f$ is invariant by the drift, so any minimizing control extended by
$u=0$ on $[\bar{t}_f,t_f]$ (where $\bar{t}_f$ denotes the min.\ time)
remains admissible. It has a cost equal to $\bar{t}_f<t_f$, hence the contradiction.
\end{proof}

\noi We set $p^0=-1$, so $H_1=|\psi|-1$.
In contrast with the minimum time case, the singularity
$\psi=0$ does not play any role in $\L^1$-minimization. In the neighbourhood of $t$
such that $\psi(t)=0$, $H_1$ is negative, so $\rho=0$. Locally, the control vanishes
and the extremal is smooth. The only effect of the singularity is a discontinuity in
the $\SS^{m-1}$ fiber over $u=0$ in which $w(t+)=-w(t-)$ (see \cite{cocv-2001}).

The important remaining singularity is $H_1=0$. As opposed to the standard
single-input case, $H_1$ is not the lift of a vector field on $X$; the properties of
the extremal flow depend on $H_0$, $H_1$, and their Poisson brackets. (See also
\S\ref{s4} for the consequences in terms of second order conditions.)
We denote by $H_{01}$ the bracket $\{H_0,H_1\}$, and so forth. The
following result is standard (see \cite{bonnard-2003a}, \eg) and accounts for the
intertwining of arcs along which $\rho=0$ (labeled $\gamma_0$) with arcs such that
$\rho=1$ (labeled $\gamma_+$). 

\begin{prpstn} \label{prop2.5}
In the neighbourhood of $z_0$ in $\{H_1=0\}$ such that $H_{01}(z_0)
\neq 0$, every extremal is locally bang-bang of the form $\gamma_0\gamma_+$ or
$\gamma_+\gamma_0$, depending on the sign of $H_{01}(z_0)$.
\end{prpstn}

\begin{proof} As $H_{01}(z_0) \neq 0$, $H_1$ must be a submersion at $z_0$, so
$\{H_1=0\}$ is locally a codimension one submanifold splitting $T^*X$ into $\{H_1<0\}$
and $\{H_1>0\}$.
Evaluated along an extremal, $H_1$ is a $\CC^1$ function of time since
\[ \dot{H}_1(t) = \{H_0+\rho(t)H_1,H_1\} = H_{01}(t). \]
Through $z_0$ passes only one extremal, and it is of the form $\gamma_0\gamma_+$ if
$H_{01}(z_0)>0$ (\resp{} $\gamma_+\gamma_0$ if $H_{01}(z_0)<0$). The bracket condition 
allows to use the implicit function theorem to prove that neighbouring extremals also
cross $\{H_1=0\}$ transversally.
\end{proof}

\noi Such switching points are termed \emph{regular} and are studied in \S\ref{s3} from
the point of view of second order optimality conditions. Besides the occurence of
$\gamma_0$ arcs resulting in the parsimony of solutions as explained in the
introduction, the peculiarity of the control setting is the existence of singular arcs
along which $H_1$ vanishes identically. On such arcs, $\rho$ may take arbitrary values
in $[0,1]$.

\begin{thrm}[Robbins \cite{robbins-1965a}]
Singular extremals are at least of order two, and minimizing singulars of order two are
contained in
\[ \{ z=(q,v,p_q,p_v) \in T^*X\ |\ V''(q)p_v^2 \geq 0,\ V'''(q)p_v^3>0 \}. \]
\end{thrm}

\begin{proof} One has $H_0=(p_q|v)-(p_v|\nabla V(q))$, and $H_1=0$ along a
singular so,
\[ 0 = H_{01} = -\frac{1}{|p_v|}(p_q|p_v) \]
along the arc.

\begin{lmm} On $T^*X$, $H_{101}=H_{1001}=0$.
\end{lmm}

\begin{proof} Computing,
\[ H_{101} = \{H_1,H_{01}\} = \{|p_v|-1,-\frac{1}{|p_v|}(p_q|p_v)\} = 0, \]
and it is standard that 
\begin{eqnarray*}
  H_{1001} &=& \{H_1,\{H_0,H_{01}\}\}\\
  &=& \{-H_{01},H_{01}\}+\{H_0,H_{101}\}\\
  &=& 0
\end{eqnarray*}
using Leibniz rule.
\end{proof}

\noi Then $0=\dot{H}_{01}=H_{001}+\rho H_{101}$ implies $H_{001}=0$ along a singular arc.
Iterating, $0=\dot{H}_{001}=H_{0001}+\rho H_{1001}$ so, by the previous lemma again,
$0=H_{0001}$. Eventually, $0=\dot{H}_{0001}=H_{00001}+\rho H_{10001}$. Set $f:=H_0$,
$g:=H_1$, $h:=-(p_q|p_v)$, so that $H_{01}=\beta h$ with $\beta=1/|p_v|$. Using
Leibniz rule, the following is clear.

\begin{lmm}
\[ (\ad^k f)(\beta h)|_{(\ad^i f)h=0,\ 0 \leq i < k} = \beta(\ad^k f)h \]
\[ \{g,(\ad^k f)(\beta h)\}|_{(\ad^i f)h=0,\ 0 \leq i \leq k} = \beta\{g,(\ad^k f)h\} \]
\end{lmm}

\noi Computing, one obtains
\[ (\ad f)h =-V''(q)p_v^2+|p_q|^2, \]
so $0=H_{001}$ implies $V''(q)p_v^2 \geq 0$, and
\[ (\ad^2 f)h =-V'''(q)(v,p_v,p_v)+4V''(q)(p_q,p_v), \]
\[ \{g,(\ad^2 f)h\} =-\frac{1}{|p_v|}V'''(q)p_v^3. \]
Through a point $z_0$ such that the last quantity does not vanish, there passes a
so-called order two singular extremal that is an integral curve of the Hamiltonian
$H_s:=H_0+\rho_s H_1$ with the dynamic feedback
\[ \rho_s :=-\frac{H_{00001}}{H_{10001}} \cdot \]
Along such a minimizing singular arc, the generalized Legrendre condition must
hold, $H_{10001} \leq 0$, so $V'''(q)p_v^3>0$.
\end{proof}

\begin{crllr} In the case of the two-body potential $V(q)=-1/|q|$ $(q \neq 0)$,
along an order two singular arc
one has either $\alpha \in (\pi/2,\alpha_0]$ or $\alpha \in [-\alpha_0,-\pi/2)$ 
where $\alpha$ is the angle of the control with the radial direction, and
$\alpha_0=\acos(1/\sqrt{3})$.
\end{crllr}

\begin{proof}
One has
\[ V'(q)p_v = \frac{(p_v|q)}{|q|^3}\,,\quad
   V''(q)p_v^2 = \frac{|p_v|^2}{|q|^3}-\frac{3(p_v|q)^2}{|q|^5}\,, \]
\[ V'''(q)p_v^3 =-\frac{9(p_v|q)|p_v|^2}{|q|^5}+\frac{15(p_v|q)^3}{|q|^7}\cdot \]
On $Q=\R^m\backslash \{0\}$, $\SS^{m-1} \ni w=p_v$ since $|p_v|=1$ along a singular arc,
so $\cos\alpha=(p_v|q)/|q|$. The condition $V''(q)p_v^2 \geq 0$ reads
$1-3\cos^2\alpha \geq 0$, and $V'''(q)p_v^3 > 0$ is fulfilled if and only if
\[ \cos\alpha(3-5\cos^2\alpha) < 0 \]
that is provided $\cos\alpha<0$ in addition to the previous condition. Hence the two
cases (in exclusion since the singular control is smooth) for the angle.
\end{proof}

\noi The existence of order two singular arcs in the two-body case results in the
well-known Fuller or chattering phenomenon \cite{marchal-1973a,zelikin-1994a}. The same
phenomenon actually persists for the restricted three body problem as is explained
in \cite{zelikin-2003a}. Altough these singular trajectories are contained in some
submanifold
of the cotangent space with codimension $>1$, their existence rules out the
possibility to bound globally the number of switchings of regular extremals described
by Proposition~\ref{prop2.5}. The next section is devoted to giving sufficient
optimality conditions for such bang-bang extremals.

\section{Sufficient conditions for extremals with regular switchings} \label{s3}
Let $X$ be an open subset of $\R^n$, $U$ a nonempty subset of $\R^m$, $f$ a
vector field on $X$ parameterized by $u \in U$, and $f^0 : X\times U \to \R$ a cost
function, all smooth.
Consider the following minimization problem with fixed final time $t_f$: Find
$(x,u):[0,t_f] \to X\times U$, $x$ absolutely continous, $u$ measurable and bounded,
such that
\[ \dot{x}(t)=f(x(t),u(t)),\quad t \in [0,t_f] \text{ (a.e)}, \]
\[ x(0)=x_0,\quad x(t_f)=x_f, \]
and such that
\[ \int_0^{t_f} f^0(x(t),u(t))\,\d t \]
is minimized.
The maximum principle asserts that, if $(\bar{x},\bar{u})$ is such a pair,
there exists an absolutely continuous lift $(\bar{x},\bar{p}):[0,t_f] \to T^*X$ and
a nonpositive scalar $\bar{p}^0$, $(\bar{p}^0,\bar{p}) \neq (0,0)$, such that a.e.\ on
$[0,t_f]$
\[ \dot{\bar{x}}(t) = \frp{H}{p}(\bar{x}(t),\bar{p}(t),\bar{u}(t))\,,\quad
   \dot{\bar{p}}(t) =-\frp{H}{x}(\bar{x}(t),\bar{p}(t),\bar{u}(t))\,, \]
and
\[ H(\bar{x}(t),\bar{p}(t),\bar{u}(t)) = \max_U H(\bar{x}(t),\bar{p}(t),\cdot) \]
where $H : T^*X \times U \to \R$ is the Hamiltonian of the problem,
\[ H(x,p,u) := p^0f^0(x,u)+pf(x,u). \]
We first assume that
\begin{itemize}
\item[(A0)] The reference extremal is normal.
\end{itemize}
Accordingly, $p^0$ can be set to $-1$.
Let $H_1,H_2: T^*X \to \R$ be two smooth functions, and denote $\Sigma:=\{H_1=H_2\}$,
$\Omega_1:=\{H_1>H_2\}$ ($\Omega_2:=\{H_2>H_1\}$, \resp)
We assume that
\begin{equation} \label{eq1}
  \max_U H(z,\cdot) = H_i,\quad z \in \Omega_i,\quad i=1,2,
\end{equation}
and follow the point of view of \cite{ekeland-1977a} that these two Hamiltonians are
competing Hamiltonians.
Let $(\bar{x},\bar{p},\bar{u})$ be a reference extremal having only one contact with
$\Sigma$ at $\bar{z}_1:=\bar{z}(\bar{t}_1)$, $\bar{t}_1 \in (0,t_f)$
($\bar{z}:=(\bar{x},\bar{p})$). We denote $H_{12}=\{H_1,H_2\}$ the Poisson bracket of
$H_1$ with $H_2$ and make the following assumption:
\begin{itemize}
\item[(A1)] $H_{12}(\bar{z}_1) > 0$.
\end{itemize}
In \cite{kupka-1987a} terms, $\bar{z}_1$ is a \emph{regular} (or \emph{normal})
switching point. This condition is called the \emph{strict bang-bang Legendre} condition
in \cite{agrachev-2002a}.
The analysis of this section readily extends to a finite number of such switchings.

\begin{lmm} $\bar{z}$ is the concatenation of the flows of $H_1$ and then $H_2$.
\end{lmm}
\begin{proof} The extremal having only one contact with $\Sigma$ at $\bar{z}_1$,
$\bar{z}(t)$ is either in $\Omega_1$ or $\Omega_2$ for $t \neq \bar{t}_1$. Because of
(\ref{eq1}), the maximization condition of the maximum principle implies that
$\bar{z}$ is given by the flow of either $H_1$ or $H_2$ on $[0,\bar{t}_1]$. In both
cases,
\[ \frac{\d}{\d t}(H_1-H_2)(\bar{z}(t))|_{t=\bar{t}_1}=-H_{12}(\bar{z}_1)<0, \]
so $H_1>H_2$ before $\bar{t}_1$ ($H_2>H_1$ after $\bar{t}_1$, \resp)
and the only possibility is an $H_1$ then $H_2$ concatenation of flows.
\end{proof}

\noi As a result of (A1), $\Sigma$ is a codimension one submanifold in the nbd of
$\bar{z}_1$, and one can define locally a function $z_0 \mapsto t_1(z_0)$ such that
$z_1(t_1(z_0),z_0)$ belongs to $\Sigma$ for $z_0$ in a nbd of $\bar{z}_0:=\bar{z}(0)$.
As we have just done, we will denote
\[ z_i(t,z_0)=e^{t \vec{H}_i}(z_0),\quad i=1,2, \]
the Hamiltonian flows of $H_1$ and $H_2$. These flows will be assumed complete for the
sake of simplicity. We will denote $'=\frpp{}{z}$ for flows. Clearly,

\begin{lmm} \label{lmm1}
\[ t_1'(z_0) = \frac{(H_1-H_2)'}{H_{12}}(z_1(t_1(z_0),z_0)) z_1'(t_1(z_0),z_0). \]
\end{lmm}

\noi One then defines locally $z_0 \mapsto z(t,z_0)=(x(t,z_0),p(t,z_0)) :=z_1(t,z_0)$
if $t \leq t_1(z_0)$,
and $z(t,z_0):=z_2(t-t_1(z_0),z_1(t_1(z_0),z_0))$ if $t \geq t_1(z_0)$.
We recall the following standard computation:
\begin{lmm} \label{lmm2}
For $t>t_1(z_0)$,
\[ \frp{z}{z_0}(t,z_0) =
   z_2'(t-t_1(z_0),z_1(t_1(z_0),z_0))(I+\sigma(z_0))z_1'(t_1(z_0),z_0) \]
with
\begin{equation} \label{eq3}
  \sigma(z_0) = \vec{H_1-H_2}\frac{(H_1-H_2)'}{H_{12}}(z_1(t_1(z_0),z_0)).
\end{equation}
\end{lmm}
\begin{proof} The derivative is equal to (arguments omitted)
\[ -\dot{z}_2 t_1' + z_2' (\dot{z}_1 t_1'+z_1'), \]
hence the result by factoring out $z_2'$ and using Lemma~\ref{lmm1} plus the fact that
the adjoint action of a flow is idempotent on its generator,
\[ (z_2'(s,z))^{-1} \vec{H}_2(z_2(s,z)) = \vec{H}_2(z),\quad (s,z) \in \R\times T^*X. \]
\end{proof}

\noi The function
\begin{equation} \label{eq3.30}
  \delta(t) := \det \frp{x}{p_0}(t,\bar{z}_0),\quad t \neq \bar{t}_1,
\end{equation}
is piecewise continuous along the reference extremal,
and we make the additional assumption that
\begin{itemize}
\item[(A2)] $\delta(t) \neq 0$, $t \in (0,\bar{t}_1) \cup (\bar{t}_1,t_f]$, and
$\delta(\bar{t}_1+)\delta(\bar{t}_1-)>0$.
\end{itemize}
This condition means that we assume disconjugacy on $(0,\bar{t}_1]$ and
$[\bar{t}_1,t_f]$ along the linearized flows of $H_1$ and $H_2$, respectively, and
that the jump (encoded by the matrix $\sigma(z_0)$) in the Jacobi fields is such that
there is no sign change in the determinant. This is exactly the condition one is able
to check numerically by computing Jacobi fields (see \cite{bonnard-2003a,oms-2012a}, \eg).
As will be clear from the proof of the
result below, geometrically this assumption is the \emph{no-fold condition} of
\cite{schattler-2002a} (no fold outside $\bar{t}_1$, no \emph{broken} fold at
$\bar{t}_1$).

\begin{thrm} Under assumptions (A0)-(A2), the reference trajectory is
a $\CC^0$-local minimizer among all trajectories with same endpoints.
\end{thrm}

\begin{proof}
We proceed in five steps.\\

\noi\emph{Step 1.}
According to (A2), $\frpp{x_1}{p_0}(t,\bar{z}_0)$ is invertible for $t \in
(0,\bar{t}_1]$; one can then construct a Lagrangian perturbation $\LL_0$ transverse to
$T^*_{x_0}X$ containing $\bar{z}_0$ such that $\frpp{x_1}{z_0}(t,\bar{z}_0)$ is
invertible for $t \in [0,\bar{t}_1]$, $t=0$ included, $\frpp{}{z_0}$ denoting the $n$
partials wrt.\ $z_0 \in \LL_0$. (See appendix~\ref{sa}.) For
$\veps>0$ small enough define
\[ \LL_1 := \{(t,z) \in \R\times T^*X\ |\ (\exists z_0 \in \LL_0) : t \in
   (-\veps,t_1(z_0)+\veps) \text{ s.t.\ } z=z_1(t,z_0)\}. \]
By restricting $\LL_0$ if necessary, $\Pi:\R\times T^*X \to \R\times X$,
$(t,z) \mapsto (t,x)$ induces a
diffeomorphism of $\LL_1$ onto its image. Similarly, (A2) implies that
\[ \frp{}{p_0} \left[ x_2(t-t_1(z_0),z_1(t_1(z_0),z_0)) \right]|_{z_0=\bar{z}_0} \]
is invertible for $t \in
[\bar{t}_1,t_f]$; restricting again $\LL_0$ if necessary, one can
assume that $\Pi$ also induces a diffeomorphism from
\[ \LL_2 := \{(t,z) \in \R\times T^*X\ |\ (\exists z_0 \in \LL_0) : t \in
   (t_1(z_0)-\veps,t_f+\veps) \]
\[ \text{ s.t.\ } z=z_2(t-t_1(z_0),z_1(t_1(z_0),z_0))\} \]
onto its image.\\

\noi\emph{Step 2.} Define $\Sigma_1:=\LL_1 \cap (\R\times\Sigma)$. As $(t,z_0) \mapsto
(t,x_1(t,z_0))$ is a diffeomorphism from $\R\times\LL_0$ onto $\Pi(\LL_1)$, there
exists an inverse function $z_0(t,x)$ such that $\Pi(\Sigma_1)=\{\psi=0\}$ with
\[ \psi(t,x) := t-t_1(z_0(t,x)). \]
Denote $\psi(t):=\psi(t,\bar{x}(t))$ the evaluation of this function along the
reference trajectory. By construction, $\dot{\psi}(\bar{t}_1-)=1>0$ and (compare
\cite{schattler-2002a})
\[ \dot{\psi}(\bar{t}_1+) = 1+\frp{t_1}{z_0}(\bar{z}_0)
   \left( \frp{x_1}{z_0}(\bar{t}_1,\bar{z}_0) \right)^{-1}
   \nabla_p(H_1-H_2)(\bar{z}_1). \]

\begin{lmm}
\begin{equation} \label{eq2}
  \delta(\bar{t}_1+) =  \delta(\bar{t}_1-) \left( 1 + \frp{t_1}{p_0}(\bar{z}_0)
   \left( \frp{x_1}{p_0}(\bar{t}_1,\bar{z}_0) \right)^{-1}
   \nabla_p(H_1-H_2)(\bar{z}_1) \right).
\end{equation}
\end{lmm}
\begin{proof} By virtue of Lemma~\ref{lmm2},
\[ \frp{x}{p_0}(\bar{t}_1+,\bar{z}_0) = \frp{x_1}{p_0}(\bar{t}_1,\bar{z}_0)
 + \nabla_p(H_1-H_2)%
   \underbrace{\frac{(H_1-H_2)'}{H_{12}}(\bar{z}_1)\frp{z_1}{p_0}(\bar{t}_1,\bar{z}_0)}_%
   {\displaystyle = \frp{t_1}{p_0}(\bar{z}_0)} \]
(the second equality coming from Lemma~\ref{lmm1}). Assumption (A2) implies
$\delta(\bar{t}_1-) \neq 0$ so, taking determinants,
\begin{eqnarray*}
\delta(\bar{t}_1+) &=& \delta(\bar{t}_1-)
\det \left( I+\left( \frp{x_1}{p_0}(\bar{t}_1,\bar{z}_0) \right)^{-1} 
\nabla_p(H_1-H_2)(\bar{z}_1)\frp{t_1}{p_0}(\bar{z}_0) \right)\\
&=& \delta(\bar{t}_1-) \left( 1+\frp{t_1}{p_0}(\bar{z}_0)
\left( \frp{x_1}{p_0}(\bar{t}_1,\bar{z}_0) \right)^{-1}
\nabla_p(H_1-H_2)(\bar{z}_1) \right)
\end{eqnarray*}
as $\det(I+x\t y)=1+(x|y)$.
\end{proof}

\noi Since $\delta(\bar{t}_1+)$ and $\delta(\bar{t}_1-)$ have the same sign, the quantity
in brackets in (\ref{eq2}) must be positive. Accordingly, $\dot{\psi}(\bar{t}_1+)>0$
as $\LL_0$ can be taken arbitrarily close to $T^*_{x_0}X$. So, locally,
$\Pi(\Sigma_1)$ is a submanifold that splits $\R\times X$ in two and, by restricting
$\LL_0$ if necessary, every extremal of the field $t \mapsto x(t,z_0)$ for $z_0 \in
\LL_0$ crosses $\Pi(\Sigma_1)$ transversally. Defining 
\[ \LL_1^- := \{(t,z) \in \R\times T^*X\ |\ (\exists z_0 \in \LL_0) : t \in
   [0,t_1(z_0)] \text{ s.t.\ } z=z_1(t,z_0)\} \]
and
\[ \LL_2^+ := \{(t,z) \in \R\times T^*X\ |\ (\exists z_0 \in \LL_0) : t \in
   [t_1(z_0),t_f] \]
\[ \text{ s.t.\ } z=z_2(t-t_1(z_0),z_1(t_1(z_0),z_0))\}, \]
one can hence piece together the restrictions of $\Pi$ to $\LL_1^-$ and $\LL_2^+$ into
a continuous bijection from $\LL_1^- \cup \LL_2^+$ into $\Pi(\LL_1^- \cup \LL_2^+)$.
By restricting to a compact neighbourhood of the graph of
$\bar{z}$, one may assume that $\Pi$ induces a homeomorphism on its image.\\

\begin{figure}[t!] \centering
\includegraphics[width=10cm]{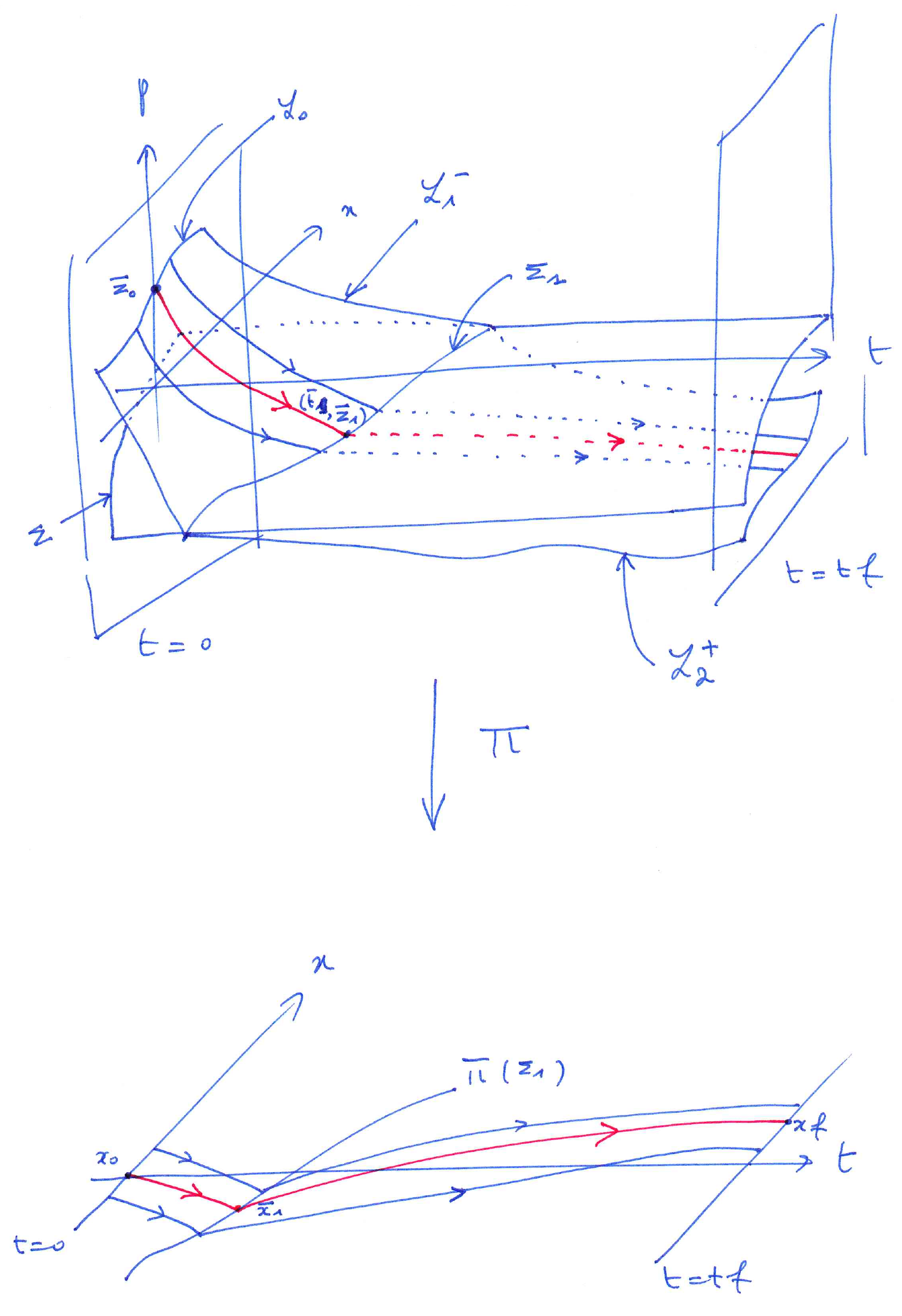}
\caption{The field of extremals.} \label{fig1}
\end{figure}

\noi\emph{Step 3.} Denote $\alpha_i := p\,\d x-H_i(z)\d t$, $i=1,2$, the Poincar\'e-Cartan
forms associated with $H_1$ and $H_2$, respectively.
To prove that $\alpha_1$ is exact on $\LL_1$,
it is enough to prove that it is closed. Indeed, if
$\gamma(s):=(t(s),z_1(t(s),z_0(s)))$ is a closed curve on $\LL_1$, it retracts
continuously on $\gamma_0(s) := (0,z_0(s))$ so that, provided $\alpha_1$ is closed,
\[ \int_\gamma \alpha_1 = \int_{\gamma_0} \alpha_1 = \int_{\gamma_0} p\,\d x = 0 \]
because $z_0(s)$ belongs to $\LL_0$ that can be chosen such that $p\,\d x$ is exact on
it. (Compare \cite[\S 17]{agrachev-2004a}.)
Similarly, to prove that $\alpha_2$ is
exact on $\LL_2$, it suffices to prove that it is closed: If
$\gamma(s):=(t(s),z_2(t(s)-t_1(z_0(s)),z_1(t_1(z_0(s)),z_0(s))))$ is a closed curve in
$\LL_2$, it readily retracts continuously on the curve
$\gamma_1(s) := (t_1(z_0(s)),z_1(t_1(z_0(s)),z_0(s)))$ in $\Sigma_1$,
which retracts continuously on $\gamma_0(s) := (0,z_0(s))$ again. Then, as $H_1=H_2$
on $\Sigma$,
\[ \int_\gamma \alpha_2 = \int_{\gamma_1} \alpha_2 = \int_{\gamma_1} \alpha_1
 = \int_{\gamma_0} \alpha_1 \]
that vanishes as before.
To prove that $\alpha_1$ is closed, consider tangent vectors at $(t,z) \in \LL_1$;
a parameterization of this tangent space is
\[ (\delta t,\vec{H}_1(z)\delta t+z_1'(t,z_0)\delta z_0),\quad
   (\delta t,\delta z_0) \in \R \times T_{z_0}\LL_0 \]
where $z_0 \in \LL_0$ is such that $z=z_1(t,z_0)$. For two such vectors $v_1$, $v_2$,
\begin{eqnarray*}
  \d\alpha_1(t,z)(v_1,v_2) &=& (\d p\wedge\d x-\d H_1(z)\d t)(v_1,v_2)\\
  &=& \d p\wedge\d x(z_1'(t,z_0)\delta z_0^1,z_1'(t,z_0)\delta z_0^2)\\
  &=& \d p\wedge\d x(\delta z_0^1,\delta z_0^2)\\
  &=& 0
\end{eqnarray*}
because $\exp(t\vec{H}_1)$ is symplectic and $\LL_0$ is Lagrangian.
Regarding $\alpha_2$, the tangent space at $(t,z) \in \LL_2$ is parameterized
according to
\[ (\delta t,\vec{H}_2(z)\delta t+z_2'(t-t_1(z_0),z_1(t_1(z_0),z_0))(I+\sigma(z_0))%
   z_1'(t,z_0)\delta z_0) \]
with $(\delta t,\delta z_0) \in \R \times T_{z_0}\LL_0$,
and where $z_0 \in \LL_0$ is such that $z=z_2(t-t_1(z_0),z_1(t_1(z_0),z_0))$.
For two such vectors $v_1$, $v_2$,
\begin{eqnarray*}
  \d\alpha_2(t,z)(v_1,v_2) &=& (\d p\wedge\d x-\d H_2(z)\d t)(v_1,v_2)\\
  &=& \d p\wedge\d x((I+\sigma(z_0))z_1'(t,z_0)\delta z_0^1,%
                     (I+\sigma(z_0))z_1'(t,z_0)\delta z_0^2)\\
  &=& \d p\wedge\d x(z_1'(t,z_0)\delta z_0^1,z_1'(t,z_0)\delta z_0^2)
\end{eqnarray*}
because $\exp(t\vec{H}_2)$ is symplectic and because
\begin{lmm}
\[ I+\sigma(z_0) \in \Sp(2n,\R). \]
\end{lmm}
\begin{proof}
For any $z \in \R^{2n}$,
\[ \t(I+Jz\t z)J(I+Jz\t z) = J-z\t z+z\t z+z(\underbrace{\t zJz}_{0})\t z = J. \]
This proves the lemma because of the definition (\ref{eq3}) of $\sigma(z_0)$.
\end{proof}

\noi One then concludes as before that $\alpha_2$ is closed using the fact that
$\exp(t\vec{H}_1)$ is symplectic and $\LL_0$ is Lagrangian.\\

\noi\emph{Step 4.} Let $(x,u):[0,t_f] \to X\times U$ be an admissible pair.
We first assume that $x$ is of class $\CC^1$ and that its graph has only one isolated
contact with $\Pi(\Sigma_1)$, at some point point $(t_1,x(t_1))$.
For $x$ close enough to $\bar{x}$ in the $\CC^0$-topology, this graph
has a unique lift $t \mapsto (t,x(t),p(t))$ in $\LL_1^- \cup \LL_2^+$. As a
gluing at $t_1$ of two absolutely continous functions, $z:=(x,p):[0,t_f] \to T^*X$ is
absolutely continous. Denote $\gamma_1$ and $\gamma_2$ the two pieces of this lift.
Denote similarly $\bar{\gamma}_1$ and $\bar{\gamma}_2$ the pieces of the graph of the 
extremal $\bar{z}$ (see Fig.~\ref{fig2}). One has
\begin{eqnarray*}
  \int_0^{t_f} f^0(x(t),u(t))\,\d t &=& \left( \int_0^{t_1}+\int_{t_1}^{t_f} \right)
  (p(t)\dot{x}(t)-H(x(t),p(t),u(t)))\,\d t\\
  &\geq& \int_0^{t_1}     (p(t)\dot{x}(t)-H_1(x(t),p(t)))\,\d t\\
  &+&    \int_{t_1}^{t_f} (p(t)\dot{x}(t)-H_2(x(t),p(t)))\,\d t\\
  &=& \int_{\gamma_1} \alpha_1 + \int_{\gamma_2} \alpha_2
\end{eqnarray*}
since $z(t)$ belongs to $\Omega_1$ for $t \in [0,t_1)$ (\resp{} to $\Omega_2$ for
$t \in (t_1,t_f]$). By connec\-ted\-ness, there exists a smooth curve $\gamma_{12} \subset
\Sigma_1$ connecting $(\bar{t}_1,\bar{z}(\bar{t}_1))$ to $(t_1,z(t_1))$; having the
same endpoints, $\gamma_1$ and $\bar{\gamma}_1 \cup \gamma_{12}$ (\resp{} $\gamma_2$
and $-\gamma_{12} \cup \bar{\gamma}_2$) are homotopic.
Since $\alpha_1$ and $\alpha_2$ are exact one forms on $\LL_1$ and $\LL_2$,
respectively,
\begin{eqnarray*}
  \int_{\gamma_1} \alpha_1 + \int_{\gamma_2} \alpha_2 &=&
  \int_{\bar{\gamma}_1 \cup \gamma_{12}} \alpha_1
+ \int_{-\gamma_{12} \cup \bar{\gamma}_2} \alpha_2\\
  &=&
  \int_{\bar{\gamma}_1} \alpha_1
+ \int_{\bar{\gamma}_2} \alpha_2\\
  &=& \int_0^{t_f} f^0(\bar{x}(t),\bar{u}(t))\,\d t
\end{eqnarray*}
since $H_1=H_2$ on $\Sigma$.\\

\begin{figure}[t!] \centering
\includegraphics[width=10cm]{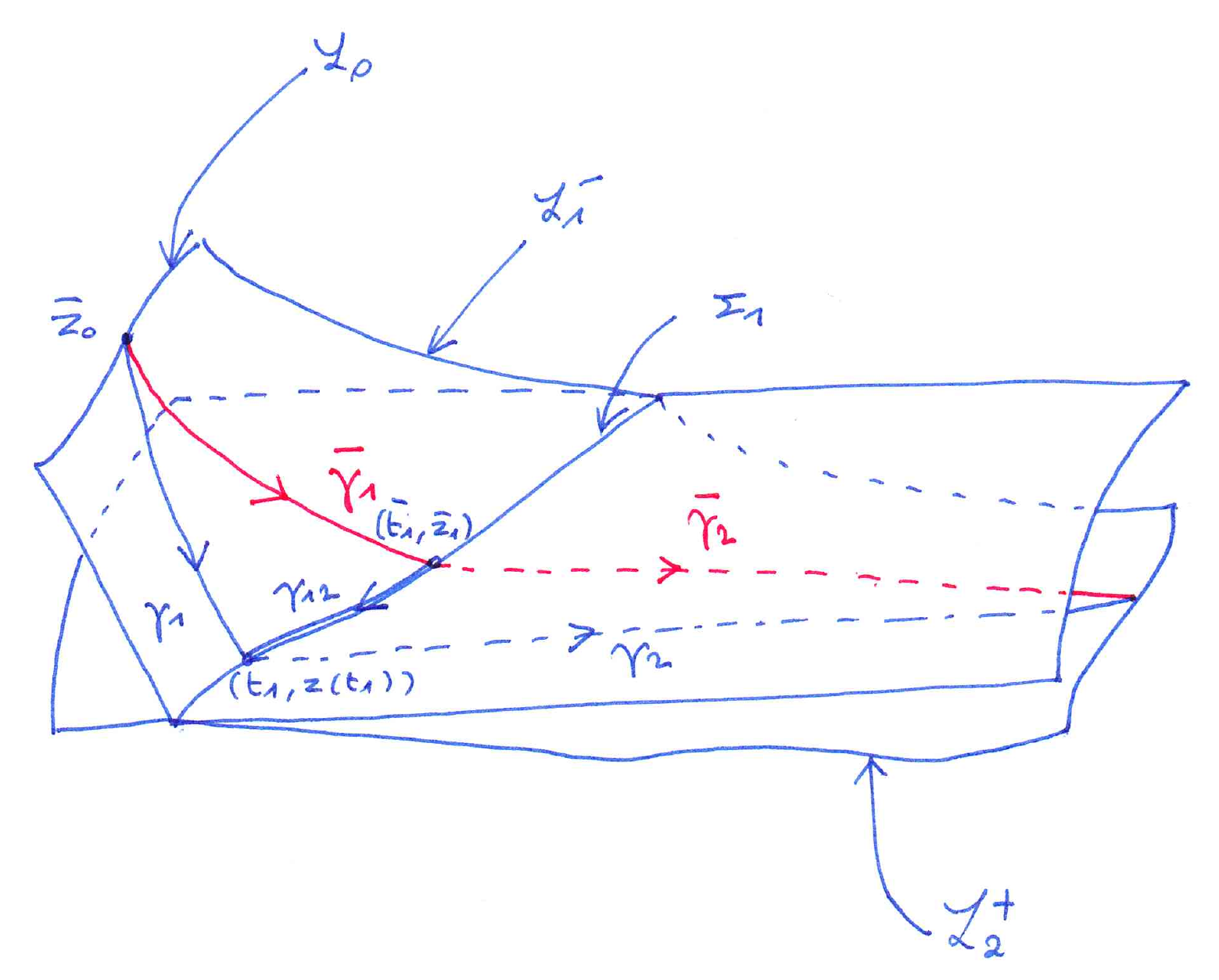}
\caption{Integration paths.} \label{fig2}
\end{figure}

\noi\emph{Step 5.} 
Consider finally an admissible pair $(x,u)$, $x$ close
enough to $\bar{x}$ in the $\CC^0$-topology. One can find $\tilde{x}$ of class $\CC^1$
arbitrarily close to $x$ in the $\W^{1,\iy}$-topology such that $\tilde{x}(0)=x_0$ and
$\tilde{x}(t_f)=x_f$. Moreover, as $\Pi(\Sigma_1)$ is a locally a smooth manifold, up
to some $\CC^1$-small perturbation one can assume that the graph of $\tilde{x}$ has only 
transverse intersections with $\Pi(\Sigma_1)$. Let $\tilde{z}:=(\tilde{x},\tilde{p})$
denote the associated lift; one has
\[ f^0(\tilde{x}(t),u(t)) =
   (\tilde{p}(t)\dot{\tilde{x}}(t)-H(\tilde{x}(t),\tilde{p}(t),u(t)))
 + \tilde{p}(t)(f(\tilde{x}(t),u(t))-\dot{\tilde{x}}(t)), \]
and the second term in the right-hand side can be made arbitrarily small when $\tilde{x}$
gets closer to $x$ in the $\W^{1,\iy}$-topology since
$(t,\tilde{z}(t))=\Pi^{-1}(t,\tilde{x}(t))$ remains bounded by continuity of the
inverse of $\Pi$.
Let then $\veps>0$; as a result of the previous discussion, there exists $\tilde{x}$
of class $\CC^1$ with same endpoints as $x$ and whose graph has only isolated contacts
with $\Pi(\Sigma_1)$ such that
\[ \int_0^{t_f} f^0(x(t),u(t))\,\d t \geq
   \int_0^{t_f} f^0(\tilde{x}(t),u(t))\,\d t - \veps, \]
\[ \int_0^{t_f} f^0(\tilde{x}(t),u(t))\,\d t \geq
   \int_0^{t_f} (\tilde{p}(t)\dot{\tilde{x}}(t)-H(\tilde{x}(t),\tilde{p}(t),u(t)))\,\d t
   - \veps. \]
One can extend straightforwardly the analysis of the previous step
to finitely many contacts with $\Pi(\Sigma_1)$, and
bound below the integral
in the right-hand side of the second inequality by the cost of the reference
trajectory. As $\veps$ is arbitrary, this allows to conclude.
\end{proof}

\section{Numerical example: The two-body potential} \label{s4}
Following \cite{gergaud-2006a}, we consider the two-body controlled problem in
dimension three. Restricting to negative energy, orbits of the uncontrolled motion are
ellipses, and the issue is to realize minimum fuel transfer between non-coplanar
orbits around a fixed center of mass. The potential is $V(q):=-\mu/|q|$ defined on
$Q := \{q \in \R^3\ |\ q \neq 0\}$, and we actually restrict to
\[ X := \{(q,v) \in TQ\ |\ |v|^2/2-\mu/|q| < 0,\ q \wedge v > 0\}. \]
(The last condition on the momentum avoids collisional trajectories and orientates
the elliptic orbits.) The constant $\mu$ is the gravitational constant that depends on
the attracting celestial body.
To keep things clear, a medium thrust case is presented below;
the final time is fixed to $1.3$ times the minimum time, approximately, which already ensures a
satisfying gain of consumption \cite{gergaud-2006a}.
In order to have fixed endpoints to perform a
conjugate point test according to \S\ref{s3} result, initial and final positions are
fixed on the orbits (fixed longitudes\footnote{Precisely, the longitude $l$ is defined
as the sum of three broken angles: $l=\Omega+\theta+\varpi$, where $\Omega$ is the
longitude of the ascending node (first Euler angle of the orbit plane with the
equatorial plane; the second Euler angle defines the inclination of the orbit),
$\theta$ is the argument of perigee (angle of the semi-major axis of
the ellipse, equal to the third Euler angle of the orbit plane), and $\varpi$ is the
true anomaly (polar angle with respect to the semi-major axis in the orbit plane).
Here, $\Omega=\theta=0$ on the initial and final orbits.}). A more relevant treatment
would leave the final longitude free (in accordance with assumption (ii) on the target
in \S\ref{s2}); this would require a focal point test that could be done much in the
same way (see, \eg, \cite{celestial-2012}). 
See Tab.~\ref{tab1} for a summary of the physical constants.

\begin{table}[t!]
\caption{Summary of physical constants used for the numerical commputation.}
\label{tab1}
\begin{tabular}{lrlr}
\\ \hline
\multicolumn{2}{l}{Gravitational constant $\mu$ of the Earth:} & 
\multicolumn{2}{r}{$398600.47\ \text{Km}^3\text{s}^{-2}$}\\
\\
Mass of the spacecraft: & $1500\ \text{Kg}$ &
Thrust: & $10\ \text{Newtons}$\\
\\
Initial perigee: & $6643\ \text{Km}$ &
Final perigee: & $42165\ \text{Km}$\\
Initial apogee: & $46500\ \text{Km}$ &
Final apogee: & $42165\ \text{Km}$\\
Initial inclination: & $0.1222\ \text{rad}$ &
Final inclination: & $0\ \text{rad}$\\
Initial longitude: & $\pi\ \text{rad}$ &
Final longitude: & $56.659\ \text{rad}$\\
\\
Minimum time: & $110.41\ \text{hours}$ &
Fixed final time: & $147.28\ \text{hours}$\\
\\
\multicolumn{2}{l}{$\mathrm{L}^1$ cost achieved (normalized):} &
\multicolumn{2}{r}{$67.617$}\\ \hline
\end{tabular}
\end{table}

As explained in \S\ref{s2}, the $\L^1$-minimization results in a competition between
two Hamiltonians: $H_0$ (coming from the drift, only), and $H_0+H_1$ (assuming the
control bound is normalized to $1$ after some rescaling). Both Hamiltonians are smooth
and fit in the framework set up in \S\ref{s3} to check sufficient optimality conditions.
Restricting to bang-bang (in the norm of the control) extremals,
regularity of the switchings is easily verified numerically,
while normality is taken care of by Proposition~\ref{prop2.1}. Then one has to check
the no-fold condition on the Jacobi fields. The optimal solution (see Fig.~\ref{fig3})
and these fields are
computed using the \texttt{hampath} software \cite{hampath}; as in
\cite{celestial-2012,gergaud-2006a},
a regularization by homotopy is used to capture the switching structure and
initialize the computation of the bang-bang extremal by single shooting. We are then
able to check condition (A2) directly on this extremal by a simple sign test
(including the jumps on the Jacobi fields at the regular switchings) on the
determinant of the fields (see Fig.~\ref{fig4}).
An alternative approach would
be to establish a convergence result as in \cite{trelat-2010b}, and to verify the
second order conditions on the sequence of regularized extremals. As underlined in
\S\ref{s1} and \S\ref{s3}, conjugate times may occur at or between switching
times. On the example treated, no conjugate point is detected on $[0,t_f]$, ensuring
strong local optimality. The extremal is then extended up to $3.5\,t_f$,
and a conjugate
point is detected about $3.2\,t_f$, at a switching point (sign change occurint at the
jump). A second test is provided Fig.~\ref{fig5}; by perturbing slightly the endpoint
conditions, one observes that conjugacy occurs not at a switching anymore, but along
a burn arc. 

\begin{rmrk} \label{rem4.1} As $H_0$ is the lift of a vector field, the determinant of
Jacobi fields is either identically zero or non-vanishing along a cost arc ($\rho=0$).
(Compare with the case of polyhedral control set; see also Corollary~3.9 in
\cite{schattler-2002a}.)
Moreover, coming from a mechanical system, the drift $F_0$
is the symplectic gradient of the energy function,
\[ E(q,v) := \frac{1}{2}|v|^2+V(q). \]
Accordingly, the $\delta x=(\delta q,\delta v)$ part of the Jacobi field (see appendix)
along an integral arc of $\vec{H}_0$ verifies
\[ \delta\dot x(t)=\vec{E}'(\bar{x}(t))\delta x(t), \]
so $\delta x$ has a constant determinant along such an arc since the associated flow
is symplectic.
In particular, the disconjugacy condition (A2) implies that the optimal solution
starts with a burn arc.
\end{rmrk}

\begin{figure}[t!] \centering
\includegraphics[width=\textwidth]{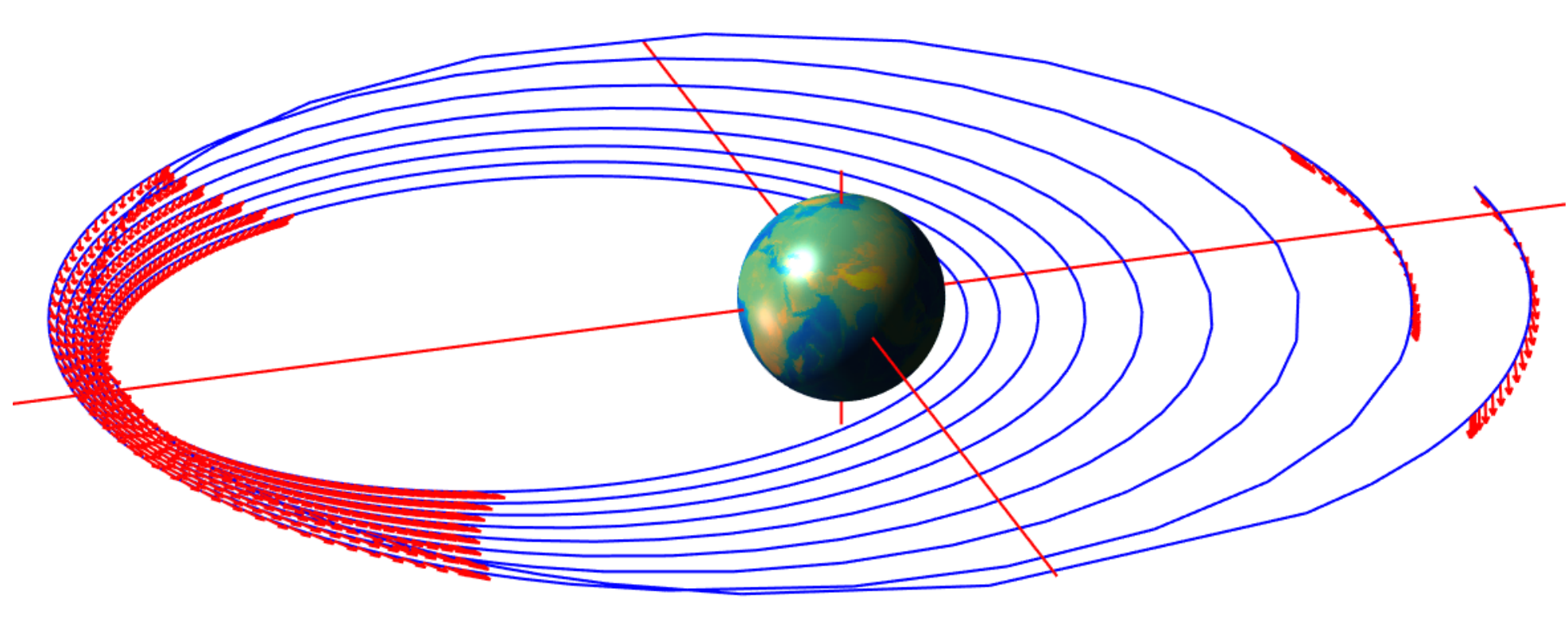}
\caption{$\mathrm{L}^1$ minimum trajectory.
The graph displays the trajectory (blue line), as well as the action of the
control (red arrows). 
The initial orbit is strongly eccentric ($0.75$) and slightly inclined
($7\ \text{degrees}$). The geostationary target orbit around the Earth is reached at
$t_f \simeq 147.28\ \text{hours}$.
The sparse structure of the control is clearly observed, with burn arcs concentrated
around perigees and apogees (see \cite{gergaud-2006a}).
The minimization leads to thrust only $46\%$ of the time.} \label{fig3}
\end{figure}

\begin{figure}[t!] \centering
\includegraphics[width=\textwidth]{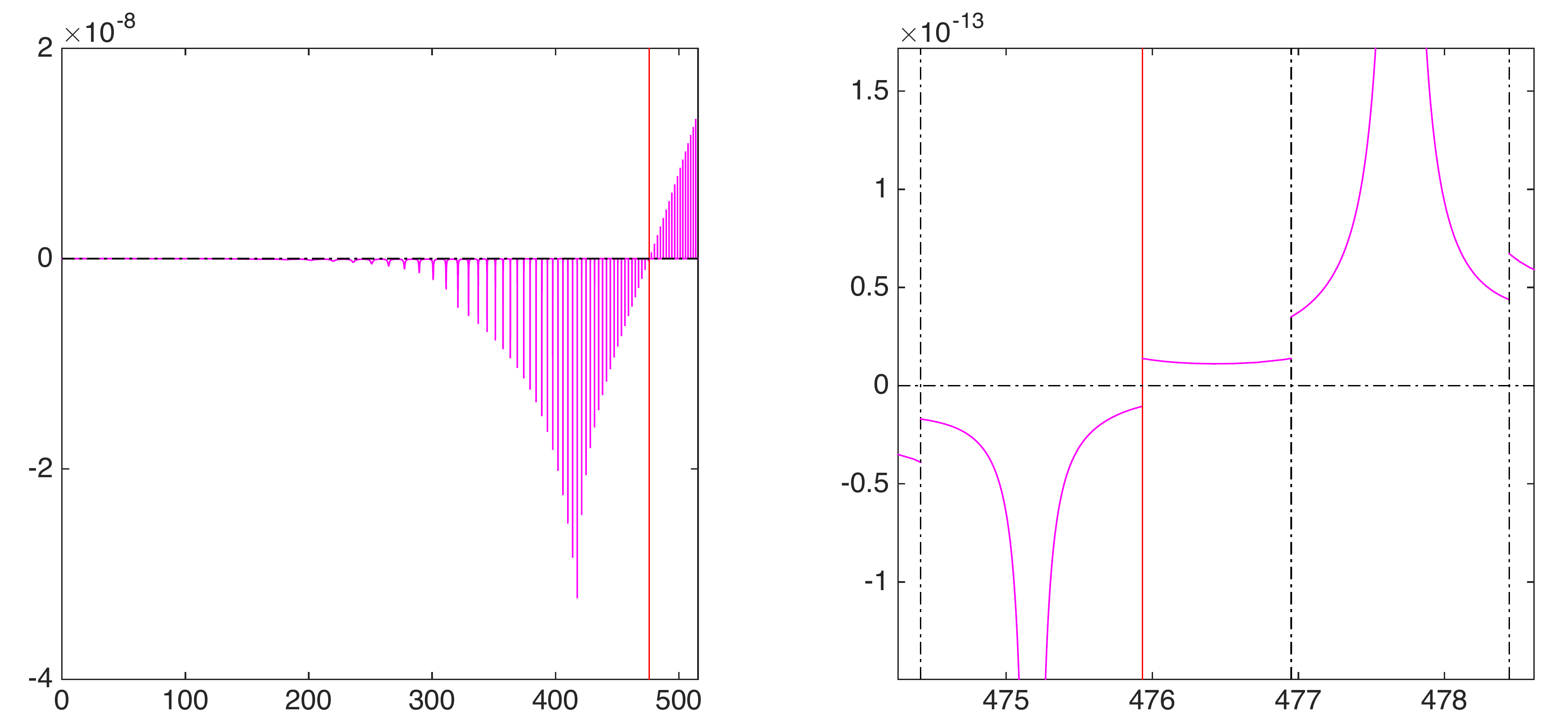}\\
\includegraphics[width=\textwidth]{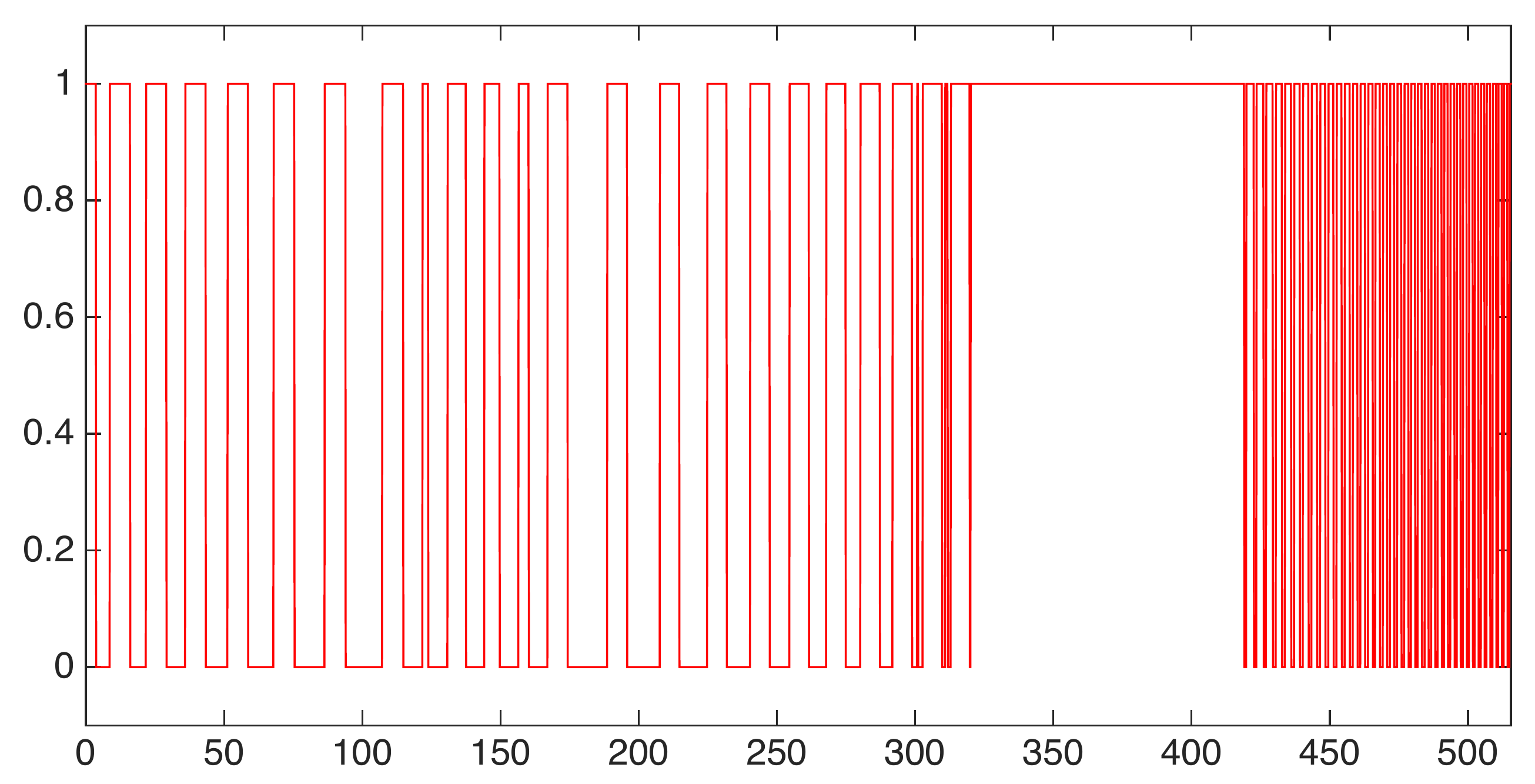}
\caption{Conjugate point test on the bang-bang $\mathrm{L}^1$-extremal extended to
$[0,3.5\,t_f]$. The value of the determinant of Jacobi fields (\ref{eq3.30})
along the extremal
is plotted against time on the upper left subgraph. The first conjugate point
occurs at $t_{1c} \simeq 475.93\ \text{hours} > t_f$;
optimality of the reference extremal on $[0,t_f]$ follows. 
On the upper right subgraph, a zoom is provided to show the jumps on the Jacobi fields
(then on their determinant) around the first conjugate time;
several jumps are observed, the first one leading to a sign change at the conjugate
time. Note that
in accordance with
Remark~\ref{rem4.1}, the determinant must be cons\-tant along the cost arcs ($\rho=0$)
provided the symplectic coordinates $x=(q,v)$ are used; this is not the case here as
the so called equinoctial elements \cite{cocv-2001}
are used for the state---hence the slight change in the determinant.
The bang-bang norm of the control, rescaled to belong to $[0,1]$ and
extended to $3.5\,t_f$, is portrayed on the lower graph.
On the extended time span, there are already more than $70$ switchings though the
thrust is just a medium one. For low thrusts, hundreds of switchings occur.} \label{fig4}
\end{figure}

\begin{figure}[t!] \centering
\includegraphics[width=\textwidth]{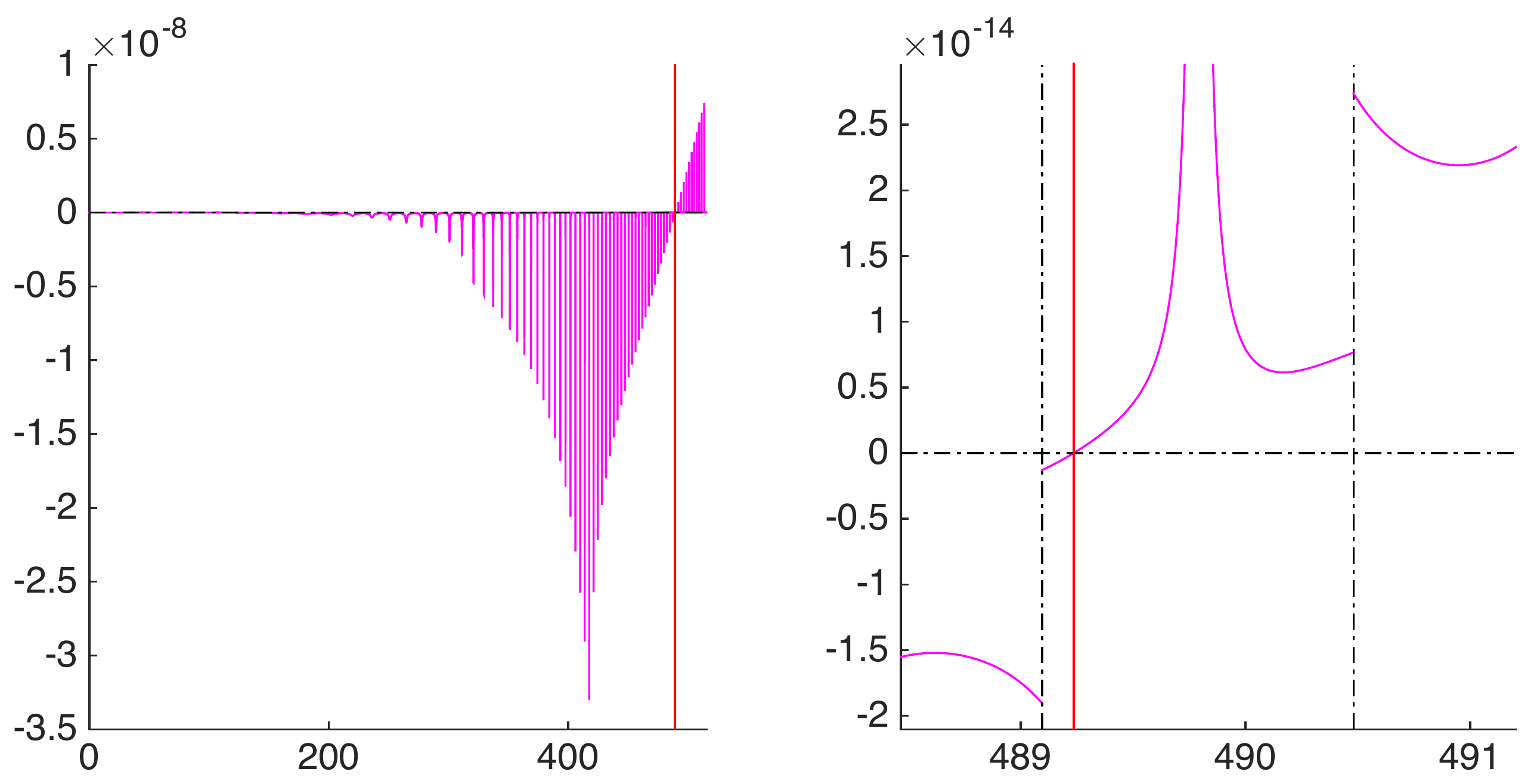}
\caption{Conjugate point test on a perturbed bang-bang $\mathrm{L}^1$-extremal extended to
$[0,3.5\,t_f]$. The value of the determinant of Jacobi fields (\ref{eq3.30})
along the extremal
is plotted against time (detail on the right subgraph). The endpoint conditions $x_0$,
$x_f$ given in
Tab.~\ref{tab1} are perturbed according to $x \leftarrow x+\Delta x$, $|\Delta x|
\simeq 1e-5$,
leading to conjugacy not at but between switching points---along a burn arc ($\rho=1$).
The first conjugate point
occurs at $t_{1c} \simeq 489.23\ \text{hours} > t_f$, ensuring again
optimality of the reference extremal on $[0,t_f]$.}
\label{fig5}
\end{figure}

\section*{Conclusion}
We have reviewed some of the particularities of $\L^1$-minimization in the control
setting. Among these, the existence of singular controls valued in the interior of the
Euclidean ball comes in strong contrast with the finite dimensional case. Moreover,
these singular extremals are at least of order two, entailing existence of
chattering \cite{zelikin-1994a}. By changing coordinates on the control, one can
reduce the system to a single control, namely the norm of the original one. This
emphasizes the role played by the Poisson structure of two Hamiltonians, the second
not the lift of a vector field; this fact accounts for the possibility of conjugacy
happening not necessarily at switching times,
as opposed to the simpler case of bang-bang controls valued in polyhedra.
Sufficient conditions for this type of extremals have been given; they
rely on a simple and numerically verifiable check on the discontinuous Jacobi fields of
the system. They are essentially equivalent to the no-fold conditions of
\cite{schattler-2002a}, formulated here in a Hamiltonian setting. The example of
$\L^1$-minimization for the three-dimensional two-body potential illustrates the
interest of the approach. Future work include the treatment of mass varying systems
(that is of maximization of the final mass) for more general problems such as the
restricted three-body one.

\appendix
\section{Sufficient conditions in the smooth case} \label{sa}
Consider the same minimization problem as in \S\ref{s3}. Suppose that
\begin{itemize}
\item[(B0)] The reference extremal is normal.
\end{itemize}
Having fixed $p^0$ to $-1$, we make the stronger assumption
that the maximized Hamiltonian is well defined and smooth, and set
\[ h(z) := \max_U H(z,\cdot),\quad z \in T^*X. \]

\begin{schlm} For almost all $t \in [0,t_f]$,
\[ h'(\bar{z}(t))=\frp{H}{z}(\bar{z}(t),\bar{u}(t)),\quad
   \nabla^2 h(\bar{z}(t))-\nabla^2_{zz}H(\bar{z}(t),\bar{u}(t)) \geq 0. \]
\end{schlm}
\begin{proof} For a.a.\ $t \in [0,t_f]$,
$h(\bar{z}(t))-H(\bar{z}(t),\bar{u}(t))=0$, while
\[ h(z)-H(z,\bar{u}(t)) \geq 0,\quad z \in T^*X, \]
by definition of $h$. Applying the first and second order necessary conditions for
optimality on $T^*X$ at $z=\bar{z}(t)$ gives the result.
\end{proof}

\noi We make the following assumption on the smooth reference extremal.
\enlargethispage{2ex}

\begin{itemize}
\item[(B1)] $\frpp{x}{p_0}(t,\bar{z}_0)$ is invertible for $t \in (0,t_f]$.
\end{itemize}

\begin{thrm} Under assumptions (B0)-(B1), the reference trajectory is
a $\CC^0$-local minimizer among all trajectories with same endpoints.
\end{thrm}

\noi Note that no Legendre type assumption is made, and that the disconjugacy condition
(B1) can be numerically verified (\eg, by a rank test while integrating the
variational system along the reference extremal).
For the sake of completeness, we provide a proof that essentially goes along the lines
of \cite[\S21]{agrachev-2004a}.

\begin{proof} For $S_0$ symmetric of order $n$, $L_0 := \{\delta x_0=S_0\delta p_0\}$
is a Lagrangian subspace of $T_{\bar{z}_0}(T^*X)$. Denote by
$\delta z=(\delta x,\delta p)$ the solution of the linearized system
\[ \delta\dot{z}(t)=\vec{h}'(\bar{z}(t))\delta z(t),\quad
   \delta z(0)=(S_0,I), \]
and set $\delta\tilde{z}(t)=(\delta\tilde{x}(t),\delta\tilde{p}(t))
:=\Phi_t^{-1}\delta z(t)$ where $\Phi_t$ is the fundamental
solution of the linearized system
\[ \dot{\Phi}_t=\frp{\vec{H}}{z}(\bar{z}(t),\bar{u}(t))\Phi_t,\quad \Phi_0=I. \]
As $\delta p(0)=\delta\tilde{p}(0)=I$,
\[ S(t):=\delta\tilde{x}(t)\delta\tilde{p}(t)^{-1} \]
is well defined for small enough $t \geq 0$. Since
\[ L_t := \exp(t\!\vec{h})'(\bar{z}(t))(L_0) \quad \text{and} \quad \Phi_t^{-1}(L_t) \]
are Lagrangian as images of $L_0$ through linear symplectic mappings, $S(t)$ must be
symmetric.

\begin{lmm} $\dot{S}(t) \geq 0$
\end{lmm}
\begin{proof} Let $t_1 \geq 0$ such that $S(t_1)$ is well defined, and let $\xi \in \R^n$.
Set
\[ \xi_0 := \delta\tilde{p}(t_1)^{-1}\xi \quad \text{and} \quad
   \delta\tilde{z}_1(t) := \delta\tilde{z}(t)\xi_0. \]
Then $\delta\tilde{z}_1(t_1)=(S(t_1)\xi,\xi)$, and
$\delta\tilde{x}_1(t)=S(t)\delta\tilde{p}_1(t)$. Differentiating the previous relation
and using $S(t)$ symmetry leads to
\[ (\dot{S}(t)\delta\tilde{p}_1(t)|\delta\tilde{p}_1(t)) =
   \omega(\delta\tilde{z}_1(t),\delta\dot{\tilde{z}}_1(t)). \]
Differentiating now
\[ \delta\tilde{z}_1(t) = \Phi_t^{-1}\delta z(t)\xi_0, \]
one gets
\[ \delta\dot{\tilde{z}}_1(t) = \Phi_t^{-1}
   ( \vec{h}'(\bar{z}(t))-\frp{\vec{H}}{z}(\bar{z}(t),\bar{u}(t)) ) \Phi_t
   \delta\tilde{z}_1(t) \]
\[ \ \ \ \ \ \ \ \ \ \ \ \ \ =
   J\t\Phi_t (\underbrace{\nabla^2 h(\bar{z}(t))-\nabla^2_{zz}H(\bar{z}(t),\bar{u}(t))}_%
     {\displaystyle \geq 0}) \Phi_t\delta\tilde{z}_1(t). \]
($J$ denotes the standard symplectic matrix.)
Evaluating at $t=t_1$, one eventually gets $(\dot{S}(t_1)\xi|\xi) \geq 0$.
\end{proof}

\noi For $S_0=0$, there is $\eta>0$ such that $S(t)$ is well defined on $[0,\eta]$, which
remains true for $S_0>0$, $|S_0|$ small enough. By the lemma before, $S_t>0$ on
$[0,\eta]$. In particular, it is an invertible matrix, which ensures that
$\Phi_t^{-1}(L_t)$ is transversal to $\ker\pi'(\bar{z}_0)$ ($\pi:T^*X \to X$ being the
canonical projection), that is $L_t$ is transversal to $\ker\pi'(\bar{z}(t))$ by virtue
of

\begin{schlm} $\Phi_t(\ker\pi'(\bar{z}_0)) = \ker\pi'(\bar{z}(t))$
\end{schlm}
\begin{proof} Note that in the linearized system defining $\Phi_t$,
\begin{eqnarray*}
  \delta\dot{x}(t) &=& \nabla^2_{xp}H(\bar{z}(t),\bar{u}(t))\delta x(t),\\
  \delta\dot{p}(t) &=&-\nabla^2_{xx}H(\bar{z}(t),\bar{u}(t))\delta x(t)
                      -\nabla^2_{px}H(\bar{z}(t),\bar{u}(t))\delta p(t),
\end{eqnarray*}
the equation on $\delta x$ is linear. Hence $\delta x(0)=0$ implies $\delta x \equiv 0$.
\end{proof}

\noi By restricting if necessary $|S_0|$, (B1) allows to assume that $\delta x(t)$
remains invertible for $t \in [\eta,t_f]$, so transversality of $L_t$ holds on
$[0,t_f]$. As a result, one can devise a Lagrangian submanifold $\LL_0$ of $T^*X$
whose tangent space at $\bar{z}_0$ is $L_0$; then
\[ \LL := \{(t,z) \in \R\times T^*X\ |\ (\exists z_0 \in \LL_0) : t \in
   (-\veps,t_f+\veps) \text{ s.t.\ } z=\exp(t\!\vec{h})(z_0) \} \]
is well defined for $\veps$ small enough, and such that $\Pi:\R\times T^*X \to \R\times
X$ induces a diffeomorphism from $\LL$ onto its image. One can moreover choose
$\LL_0$ such that $p\,\d x$ is not only closed but an exact form on it, in order that
the Poincar\'e-Cartan form $p\,\d x-h(z)\d t$ is exact on $\LL$. This, together with
assumption (B0), allows to conclude as
usual that the reference trajectory is optimal with respect to $\CC^0$-neighbouring
trajectories with same endpoints.\end{proof}

\end{document}